\documentclass[a4paper,11pt]{article}
\usepackage{array}
\usepackage{amsmath,amscd,amssymb,amsthm} 
\usepackage{latexsym}
\usepackage{epsfig}
\usepackage{enumerate}
\usepackage{float}
\usepackage{systeme}
\usepackage{mathdots}
\usepackage{xcolor}
\usepackage{hyperref}
\hypersetup{colorlinks=true,linkcolor=blue, citecolor=blue}

\theoremstyle{plain}
\newtheorem{lemma}{Lemma}[section]

\newtheorem{proposition}[lemma]{Proposition}
\newtheorem{theorem}[lemma]{Theorem}
\newtheorem{corollary}[lemma]{Corollary}
\newtheorem{definition}[lemma]{Definition}
\theoremstyle{definition}

\theoremstyle{remark}
\newtheorem{remark}[lemma]{Remark}

\newcommand{\PP}{\mathbb P}
\newcommand{\ZZ}{\mathbb Z}
\newcommand{\cB}{\mathcal B}
\newcommand{\cE}{\mathcal E}
\newcommand{\cF}{\mathcal F}

\newcommand{\cO}{\mathcal O}
\newcommand{\cV}{\mathcal V}

\newcommand{\To}{\longrightarrow}
\newcommand{\Pic}{\mathop{\mathrm {Pic}}\nolimits}
\newcommand{\Proj}{\mathop{\null\mathrm {Proj}}\nolimits}

\newcommand{\Sym}{\mathop{\mathrm {Sym}}\nolimits}

\newcommand{\Reg}{\mathop{\mathrm {Reg}}\nolimits}

\newcommand{\bdp}{\mathbf p}

\newcommand{\bdh}{\mathbf h}
\newcommand{\bdf}{\mathbf f}
\newcommand\tw[2]{\!\left[\begin{smallmatrix}{#1}\\{#2}\end{smallmatrix}\right]}

\newcommand{\Hyp}{{\bdh}}
\newcommand{\fib}{{\bdf}}

\newenvironment{sistema}%
  {\left\lbrace\begin{array}{@{}l@{}}}%
  {\end{array}\right.}
\makeatletter
\renewcommand\@makefntext[1]{%
\setlength\parindent{1em}%
\noindent
\mbox{\@thefnmark}{#1}}
\makeatother
\begin{document}
\title{\textbf{Ulrich bundles on smooth toric threefolds with Picard number $2$}}
\author{Debojyoti Bhattacharya and Francesco Malaspina}
\date{\vspace{-5ex}}

\maketitle
\def\thefootnote{\null}

\begin{abstract}\noindent
In this paper, we study Ulrich bundles on smooth toric threefolds with Picard number$~2$, namely $\PP(\cO_{\PP^2}(a_0)\oplus\cO_{\PP^2}(a_1))$.
We construct resolutions and monads for Ulrich bundles of arbitrary rank, and  provide
explicit examples together with a complete classification of those arising as pullbacks
from $\mathbb{P}^2$. As a consequence, we also show that these varieties are Ulrich wild.\\
\end{abstract}

\section*{Introduction}\label{sect:intro}

Arithmetically Cohen--Macaulay (aCM) vector bundles have attracted significant
attention in recent years, due to their role in the classification of vector bundles
on various projective varieties. Their importance lies in the fact that aCM bundles
provide a measure of the complexity of the underlying variety. Moreover, they
constitute a crucial ingredient in the study of arbitrary vector bundles, as
shown in~\cite{ES}.

A special subclass of aCM bundles is given by \emph{Ulrich bundles}, namely those
achieving the maximum possible  number of minimal generators. Originally studied in connection with the computation of Chow
forms, Ulrich bundles are conjectured to exist on any projective variety (see~\cite{ES}).
They are characterized by the linearity of the minimal graded free resolution of the
module of global sections over the polynomial ring. Several recent works focus on
Ulrich bundles on significant varieties (see \cite
{A, ACM, CCHMW, CM1, CM2}). See also \cite{UB} for a comprehensive survey. For instance, in~\cite{CMP0} the
authors construct families of arbitrary dimension on Segre varieties (except
$\mathbb{P}^1 \times \mathbb{P}^1$). In~\cite{FM} it is shown that every Ulrich bundle on a 
rational normal scroll of dimension two arises as an extension of two direct sums of line
bundles; moreover, in the case of quartic scrolls a complete classification of aCM
bundles is obtained. In ~\cite{CFA},  the moduli space of stable Ulrich bundles of rank $r$ and determinant $\mathcal O_X(r)$  on any smooth Fano threefold X of certain specific index is considered and it is shown that the moduli space is  smooth of dimension $r^2+1$. In ~\cite{FF}, investigation concerning the existence of Ulrich bundles and Ulrich complexity is carried out on some suitable threefold scrolls over Hirzebruch surfaces $\mathbb F_{e}$ for any $e \geq 0$.   In~\cite{AHMP}, Ulrich bundles of arbitrary rank on smooth projective
varieties of minimal degree and of any dimension are  characterized by
means of a special type of filtration. As a consequence, their moduli spaces are
zero-dimensional.

The case $\mathbb{P}^2 \times \mathbb{P}^1$ is particularly interesting: in fact, there exist only finitely
many aCM bundles that are not Ulrich (see \cite{FMS}). Veronese varieties have been
investigated in~\cite{ES}, where it is shown that every Ulrich bundle has natural
cohomology and arises from a graded resolution. These results have been
generalized to Segre--Veronese varieties in~\cite{M}.

The aim of this article is to obtain an analogous result for smooth projective
threefolds $X$ obtained as projectivisations of sums of line
bundles on $\PP^2$, i.e.  $X=\PP(\cO_{\PP^2}(a_0)\oplus\cO_{\PP^2}(a_1))$ with $0<a_0\leq a_1$. In
fact it was shown in~\cite{Kl} that these are exactly the smooth toric threefolds with Picard number two.\\


We now present the main results and the organization of the paper as follows:\\

In section $\S \ref{sect:scrolls}$,  we collect few definitions and facts related to the basic objects (e.g. aCM, Ulrich sheaf, objects in the derived category of coherent sheaves, Beilinson theorems) concerned which will be required in the subsequent sections.\\

In section $\S \ref{S2}$, in the first part i.e. in subsection \ref{S21}, we explicitly describe the base variety $X$ i.e. smooth toric threefolds with Picard number $2$ and document some crucial facts (e.g. S.E.S's, cohomology of certain sheaves) on them which will be useful later. In the second part i.e. in subsection \ref{sect:splitting}, we prepare the main technical ingredient to prepare the Beilinson table to construct  resolution for Ulrich bundles on $X$. To be more precise, for an Ulrich bundle $\cE$ on $X$, we establish cohomology vanishing results (in the spirit of~\cite{ES} for Veronese varieties) in a suitable range (see Proposition \ref{riv2}) and we then compute the cohomology of $\cE$
tensored with the pullbacks of the bundles $\Omega^{1}_{\mathbb{P}^2}$ with appropriate twist (see Proposition \ref{riv2}). As a consequence, we deduce that every Ulrich bundle $\cE$ on $X$ is regular in the sense of the Castelnuovo--Mumford regularity defined in~\cite{MS} (See Remark \ref{rem}$(i)$). We also  use the generalized Hoppe's criterion over polycyclic varieties as in \cite{Gauge} to obtain improvements of the already obtained vanishing results (See Remark \ref{rem}$(ii)$) which are useful in obtaining resolution, monadic description of Ulrich bundles on some special subclasses of the base variety under consideration.\\

In section $\S \ref{Res}$,  with the cohomological vanishing properties of Ulrich bundles over $X$ at hand,  we establish the main theorem of this paper i.e. construction of a resolution of Ulrich bundles on $X$. To state our main result, let $\Hyp$
and $\fib$ be the generating classes in $\Pic X$ (for detailed description see subsection \ref{S21}) and let 
$\cF\tw{a}{b}:=\cF(a\Hyp+b\fib)$ for  a sheaf $\cF$ on $X$. For an Ulrich bundle $\cE$ on $X$, let :
\begin{itemize}
    \item $a_i^{j,k}:=h^{i}(\cE\tw{j}{k})$.\

    \item $b_i^{j,k}:=h^{i}(\cE \otimes \Omega_{\pi}\tw{j}{k})$, where, $\Omega_{\pi}$ is the pullback of $\Omega^1_{\mathbb P^2}$ (see \eqref{eq:34}).\
\end{itemize}
Thanks to a Beilinson type spectral sequence (see Theorem \ref{use}) constructed with suitable exceptional collections (\eqref{col5} and the dual \eqref{col6}), we prove the following theorem, which gives a resolution (see Theorem \ref{main}):


\begin{equation*}
0\to \cO_X\tw{-1}{-1}^{\oplus a_0^{-1,c-2}}\to \cO_X\tw{-1}{0}^{\oplus b_0^{-1,c}} \bigoplus \cO_X\tw{0}{-2}^{\oplus a_0^{0,-1}}\to 
\end{equation*}
$$ \to \cO_X\tw{-1}{1}^{\oplus a_0^{-1,c-1}} \bigoplus \cO_X\tw{0}{-1}^{\oplus b_0^{0,1}} \to \cO_X^{\oplus a_0^{0,0}} \to \cE \to 0$$

We also obtain resolution of Ulrich bundles on $X$ with low values of $c$ (see Proposition \ref{lowc}) and monadic description for certain Ulrich bundles on $X$ (see Remark \ref{monad}).\\

In section $\S \ref{S4}$, we focus on providing examples of Ulrich bundles on $X$ and establishing the Ulrich wildness of $X$. Thanks to the study of Ulrich bundles on Veronese surfaces (as in \cite{Genc} and \cite{Costa}), we obtain a complete characterization of the Ulrich bundles of arbitrary rank on $X$ which are pullbacks from $\mathbb P^2$. To be more precise, for any positive integer $d$, if we denote $(\mathbb P^2, dH)$ to be the Veronese  surface of degree $d^2$, let $G_\pi:=\pi^{*}(G)$ for some bundle $G$ on $\mathbb P^2$ and $c=a_0+a_1$, then we prove that (see Proposition \ref{pullback}):\\
$G_\pi\tw{a}{b}$ is Ulrich on $X$ for some $G$ on $\mathbb P^2$ and $(a,b) \in \mathbb Z^2$ if and only if one of the following conditions is satisfied:

$(i)$  $a=0$, $a_0=a_1$, $b=c-a_0$ and $G$ is Ulrich on $(\mathbb P^2, a_0H)$.\



$(ii)$  $a=1$, $b=-c$ and  $G$ is Ulrich on $(\mathbb P^2, cH)$.\

$(iii)$ $a=2$, $a_0=a_1$, $b=-2a_0$ and $G$ is Ulrich on $(\mathbb P^2, a_0H)$.\

As a consequence of this understanding of Ulrich pullbacks on $X$, we get a complete classification of Ulrich line bundles on $X$ (see Corollary \ref{line}), Ulrichicity of twisted $\Omega_\pi$ on $X$ (see Corollary \ref{cotangent}) and the Ulrich wildness of $X$ (see Corollary \ref{wild}). Finally, we discuss  the existence of even rank Ulrich bundles which are not pullbacks (see Remark \ref{nonpullback}) and end the paper with two open questions.\


\section{Preliminaries}\label{sect:scrolls}

In this section, we recall the definitions and theorems related to some of the main objects of this paper. We work with projective variety over an algebraically closed field of characteristic $0$. We start with the definition of aCM and Ulrich sheaf on a projective variety. If $\cE$ is a coherent sheaf on a projective variety $X$, we simply write $H^i(\cE):=H^i(X, \cE)$ and $h^i(\cE)=\text{dim}(H^i(\cE))$.\

\begin{definition}
A coherent sheaf $\cE$ on a projective variety $X$ with a fixed ample line bundle $\cO_X(1)$ is called {\it arithmetically Cohen-Macaulay} (for short, aCM) if it is locally Cohen-Macaulay and $H^i(\cE(t))=0$ for all $t\in \ZZ$ and $i=1, \ldots, \dim (X)-1$.
\end{definition}



\begin{definition}
For an {\it initialized} coherent sheaf $\cE$ on $X$, i.e. $h^0(\cE(-1))=0$ but $h^0(\cE)\ne 0$, we say that $\cE$ is an {\it Ulrich sheaf} if it is aCM and $h^0(\cE)=\deg (X)\mathrm{rank}(\cE)$.\
\end{definition}

Ulrich bundles on a smooth projective variety can be characterized in many different ways:

\begin{proposition}\label{Ulrich}
Let $(X, \mathcal O_X(1))$  be a $n$-dimensional smooth projective variety and $\cE$ be an initialized vector bundle on $X$. Then the following conditions are equivalent:\

$(i)$ $\cE$ is Ulrich.\

$(ii)$ $H^i(\cE(-t))=0$ for $i \geq 0$ and $1 \leq t \leq n$.\

$(iii)$ $H^i(\cE(-i))=0$ for $i > 0$ and $H^i(\cE(-i-1))=0$ for $i <n$.\

$(iv)$ For some (respectively all) finite linear projections $\pi: X \to \mathbb P^n$, the vector bundle $\pi_{*}(\cE)$ is the trivial vector bundle $\mathcal O_{\mathbb P^n}^{\oplus t}$ for some $t$.\
\end{proposition}

\begin{proof}
See \cite{ES}, Proposition $2.1$ and \cite{UB}, Theorem $3.2.9$.\
\end{proof}

Throughout the paper, we will stick to the equivalent definition as in Proposition \ref{Ulrich}$(ii)$ for Ulrich bundles.\\

Next, we recall the definitions concerning the objects in the derived category of coherent sheaves over a smooth projective variety e.g. the notion of a full strong exceptional collection, left and right mutation functor, left and right dual collections, which are useful for preparing the Beilinson table in section \ref{Res}.\\

Given a smooth projective variety $X$, let $D^b(X)$ be the the bounded derived category of coherent sheaves over $X$. An object $E \in D^b(X)$ is called {\it exceptional} if $Ext^\bullet(E,E) = \mathbb C$.
A set of exceptional objects $\langle E_0, \ldots, E_n\rangle$ is said to be an {\it exceptional collection} if $Ext^\bullet(E_i,E_j) = 0$ for $i > j$. An exceptional collection is called {\it full} when $Ext^\bullet(E_i,A) = 0$ for all $i$ implies $A = 0$, or equivalently when $Ext^\bullet(A, E_i) = 0$ does the same. Moreover, An exceptional collection is further called strong exceptional collection
if $Ext^k(E_i, E_j ) = 0$ for all $i, j$ and $k > 0$.\

\begin{definition}\label{def:mutation}
Let $E$ be an exceptional object in $D^b(X)$.
Then there are functors $\mathbb L_{E}$ and $\mathbb R_{E}$ fitting in distinguished triangles
$$
\mathbb L_{E}(T) 		\to	 Ext^\bullet(E,T) \otimes E 	\to	 T 		 \to	 \mathbb L_{E}(T)[1]
$$
$$
\mathbb R_{E}(T)[-1]	 \to 	 T 		 \to	 Ext^\bullet(T,E)^* \otimes E	 \to	 \mathbb R_{E}(T)	
$$
The functors $\mathbb L_{E}$ and $\mathbb R_{E}$ are called respectively the \emph{left} and \emph{right mutation functor}.
\end{definition}


The collections defined by
\begin{align*}
E_i^{\vee} &= \mathbb L_{E_0} \mathbb L_{E_1} \dots \mathbb L_{E_{n-i-1}} E_{n-i};\\
^\vee E_i &= \mathbb R_{E_n} \mathbb R_{E_{n-1}} \dots \mathbb R_{E_{n-i+1}} E_{n-i},
\end{align*}
are again full and exceptional. They are said to be  the \emph{right} and the \emph{left dual} collections respectively. The dual collections are characterized by the following property; see \cite[Section 2.6]{GO}.
\begin{equation}\label{eq:dual characterization}
Ext^k(^\vee E_i, E_j) = Ext^k(E_i, E_j^\vee) = \left\{
\begin{array}{cc}
\mathbb C & \textrm{\quad if $i+j = n$ and $i = k $} \\
0 & \textrm{\quad otherwise}
\end{array}
\right.
\end{equation}

\vspace{5mm}

Equipped with the above notions in $D^b(X)$, we end this section by noting down the Beilinson spectral theorems, a crucial ingredient for obtaining resolution for Ulrich bundles in section \ref{Res}.\

\begin{theorem}[Beilinson spectral sequence]\label{thm:Beilinson}
Let $X$ be a smooth projective variety along with a full exceptional collection $\langle E_0, \ldots, E_n\rangle$ of objects in $D^b(X)$. Then for any object $A$ in $D^b(X)$, there is a spectral sequence
with the $E_1$-term given by
\[
E_1^{p,q} =\oplus_{r+s=q} Ext^{n+r}(E_{n-p}, A) \otimes \mathcal H^s(E_p^\vee )
\]
which is functorial in $A$ and converges to $\mathcal H^{p+q}(A)$.
\end{theorem}

Note that the statement and the proof of the above theorem  can be found  in  \cite[Corollary 3.3.2]{RU}, in \cite[Section 2.7.3]{GO} and in \cite[Theorem 2.1.14]{BO}.\


 Now assume that the full exceptional collection  $\langle E_0, \ldots, E_n\rangle$ contains only pure objects of type $E_i=\mathcal E_i^*[-k_i]$ with $\mathcal E_i$ a vector bundle for each $i$, and also the right dual collection $\langle E_0^\vee, \ldots, E_n^\vee\rangle$ consists of coherent sheaves. Then the Beilinson spectral sequence is much simpler as
\[
E_1^{p,q}=Ext^{n+q}(E_{n-p}, A) \otimes E_p^\vee=H^{n+q+k_{n-p}}(\mathcal E_{n-p}\otimes A)\otimes E_p^\vee.
\]

Observe that the grading in this spectral sequence applied for the projective space is slightly different from the grading of the usual Beilison spectral sequence, due to the existence of shifts by $n$ in the index $p,q$. Indeed, the $E_1$-terms of the usual spectral sequence are $H^q(A(p))\otimes \Omega^{-p}(-p)$, which are $0$ for positive $p$. To restore the order one needs to slightly change the gradings of the spectral sequence of the Theorem \ref{thm:Beilinson}. If we replace, in the expression
\[
E_1^{u,v} = \mathrm{Ext}^{v}(E_{-u},A) \otimes E_{n+u}^\vee=
H^{v+k_{-u}}(\mathcal E_{-u}\otimes A) \otimes \mathcal F_{-u}
\]
$u=-n+p$ and $v=n+q$, so that the fourth quadrant is mapped to the second quadrant, we obtain the following version (see \cite{AHMP}) of the Beilinson spectral sequence as follows:

\begin{theorem}\label{use}
Let $X$ be a smooth projective variety with a full exceptional collection
$\langle E_0, \ldots, E_n\rangle$
where $E_i=\mathcal E_i^*[-k_i]$ with each $\mathcal E_i$ a vector bundle and $(k_0, \ldots, k_n)\in \ZZ^{\oplus n+1}$ such that there exists a sequence $\langle F_n=\mathcal F_n, \ldots, F_0=\mathcal F_0\rangle$ of vector bundles satisfying
\begin{equation}\label{order}
\mathrm{Ext}^k(E_i,F_j)=H^{k+k_i}( \mathcal E_i\otimes \mathcal F_j) =  \left\{
\begin{array}{cc}
\mathbb C & \textrm{\quad if $i=j=k$} \\
0 & \textrm{\quad otherwise}
\end{array}
\right.
\end{equation}
i.e. the collection $\langle F_n, \ldots, F_0\rangle$ labelled in the reverse order is the right dual collection of $\langle E_0, \ldots, E_n\rangle$.
Then for any coherent sheaf $A$ on $X$ there is a spectral sequence in the square $-n\leq p\leq 0$, $0\leq q\leq n$  with the $E_1$-term
\[
E_1^{p,q} = \mathrm{Ext}^{q}(E_{-p},A) \otimes F_{-p}=
H^{q+k_{-p}}(\mathcal E_{-p}\otimes A) \otimes \mathcal F_{-p}
\]
which is functorial in $A$ and converges to
\begin{equation}
E_{\infty}^{p,q}= \left\{
\begin{array}{cc}
A & \textrm{\quad if $p+q=0$} \\
0 & \textrm{\quad otherwise.}
\end{array}
\right.
\end{equation}
\end{theorem}

\begin{remark}\label{rembeil}
It is possible to state a stronger version of the Beilinson's theorem (see \cite{ottanc}, \cite{Be} for $\mathbb{P}^N$ and \cite{AO3} for the projectivized of a direct sum of line bundles over $\mathbb{P}^N$). Let us consider $X = \PP\cV$ with $\cV=\oplus_{i=0}^t
\cO_{\PP^m}(a_i)$ and let $A$ be a coherent sheaf on $X$. Let $(E_0,\dots,E_n)$ be a full exceptional collection and $(F_n,\dots,F_0)$ its right dual collection. Using the notation of the  Theorem \ref{use}, if $(F_n,\dots,F_0)$ is strong, then there exists a complex of vector bundles $L^\bullet$ such that
\begin{enumerate}
\item $H^k(L^\bullet)=
\begin{cases}
A \ & \text{if $k=0$},\\
0 \ & \text{otherwise}.
\end{cases}$
\item $L^k=\underset{k=p+q}{\bigoplus}H^{q+k_{-p}}(A\otimes E_{-p})\otimes F_{-p}$ with $0\le q \le n$ and $-n\le p \le 0$.
\end{enumerate}
\end{remark}

\section{\hspace{-4mm}Cohomology \hspace{-1mm} of Ulrich bundles on \texorpdfstring{$\PP(\cO_{\PP^2}(a_0)\oplus\cO_{\PP^2}(a_1))$}{O(P2)(a0) + O(P2)(a1)}}\label{S2}


The purpose of this section is twofold and therefore, for convenience, we divide this section into two subsections. In the first section, we describe the base variety i.e. smooth toric threefolds having Picard number $2$ and record some useful facts/results on them  and in the second subsection, we study the cohomological vanishing properties of Ulrich bundles of arbitrary rank on these threefolds.\\

\subsection{Smooth toric threefolds with Picard number 2}\label{S21}

We fix a decomposible vector bundle $\cV=
\cO_{\PP^2}(a_0) \oplus \cO_{\PP^2}(a_1)$ of rank $2$ on $\PP^2$: we assume throughout that
$a_0\le a_1$. The associated projective space bundle
$X:=\PP\cV$ is by definition $\Proj(\Sym \cV)$, adopting the
notational conventions of \cite[Section~II.7]{Hartshorne}. The
associated line bundle $\cO_X(1)$ is relatively ample over $\PP^2$,
and is ample on $X$ if $a_0>0$. We put $c:= a_{0}+a_{1}$ and we let
$\pi\colon \PP (\cV) \to \PP^2$ be the projection. We denote by $\Hyp$
and $\fib$ the classes in $\Pic X$ of $\cO_{\PP(\cV)}(1)$
and the pullback $\pi^*\cO_{\PP^2}(1)$, respectively.  We note that the Chow ring of $X$ is given by: 
$$A(X)\cong\mathbb Z[\Hyp, \fib]/(\fib^3,\Hyp^2-c.\Hyp .\fib+a_0.a_1\fib^2),
$$ 
Recall that  $\Hyp^3=c^2-a_0.a_1$, $\Hyp^2.\fib=c.\Hyp. \fib^2$ and $\Hyp. \fib^2$ is the class of a point.\\





We fix the following notation and convention from now on:\

\begin{itemize}

\item For conciseness, if $\cF$ is a sheaf on $X$ we will often write
$\cF\tw{a}{b}:=\cF(a\Hyp+b\fib)$.\


\item With this notation, we have the canonical sheaf $\omega_X \cong \cO_X\tw{-2}{c-3}$.\

\item For our purpose, since we are considering the smooth case, from now on, we consider the above varieties $X$ with $a_0 >0$.\

\item When we say $\cE$ is an Ulrich bundle on $X$, we mean $\cE$ is an Ulrich bundle with respect to $\cO_X\tw{1}{0}$ (which is also a very ample line bundle if $a_0>0$).\

\end{itemize}

The following  easy lemma is useful for computation.
\begin{lemma}\label{lem:vanishing}
Let $X$ be as above. Then the following holds:
\begin{enumerate}[(i)]
\item $H^i(X, \cO_X\tw{a}{b}) \cong H^i(\PP^2, \Sym^a\cV \otimes
  \cO_{\PP^2}(b))$ if $a\ge 0$;
\item $H^i(X, \cO_X\tw{a}{b})) =0$ if $a=-1$;
\item $H^i(X, \cO_X\tw{a}{b}) \cong H^{3-i}(\PP^2, \Sym^{-a-2}\cV
  \otimes \cO_{\PP^2}(c-b-3))$ if $a<-1$.
\end{enumerate}
\end{lemma}
\begin{proof}
See \cite[Exercise III.8.4]{Hartshorne}.
\end{proof}

To obtain a complete understanding of Ulrich pullbacks on $X$, we need the following generalized version of Lemma \ref{lem:vanishing} for vector bundles.\

\begin{lemma}\label{lem:vanishingbundle}
Let $X$ be as above and $G_\pi:=\pi^{*}(G)$ for some $G$ on $\mathbb P^2$. Then the following holds:\
\begin{enumerate}[(i)]
\item $H^i(X, G_{\pi}\tw{a}{b}) \cong H^i(\PP^2, \Sym^a\cV \otimes
  G(b))$ if $a\ge 0$;
\item $H^i(X, G_{\pi}\tw{a}{b})) =0$ if $a=-1$;
\item $H^i(X, G_{\pi}\tw{a}{b}) \cong H^{3-i}(\PP^2, \Sym^{-a-2}\cV
  \otimes G^{*}(c-b-3))$ if $a<-1$.
\end{enumerate}
\end{lemma}

\begin{proof}
The assertion follows from \cite[Exercise III.8.4]{Hartshorne}.\\
\end{proof}

Next, we note down the following short exact sequences (S.E.S's) on our threefold $X$ which will be crucial for obtaining the cohomology vanishing properties in the next subsection:\\

We have the dual of relative Euler exact sequence on $X$ given by: \

\begin{equation}\label{eq:33}
0\To \cO_X\tw{-1}{c} \To \cB:=
\bigoplus_{i=0}^1\cO_X\tw{0}{a_i} \To \cO_X\tw{1}{0}\To 0,
\end{equation}

Let $\Omega_{\pi}:=\pi^{*}\Omega^{1}_{\mathbb P^2}$. Then we have the pullback (under $\pi$) of the Euler exact sequence on $\mathbb P^2$ given by:

\begin{equation}\label{eq:34}
0\To \Omega_{\pi}\To
\cO_X\tw{0}{-1}^{\oplus {3}} \To \cO_X \To 0,\
\end{equation}
 and the pullback (under $\pi$) of the dual of the Euler exact sequence on $\mathbb P^2$ given by:

\begin{equation}\label{eq:35}
0\To \cO_X\tw{0}{-3} \To
\cO_X\tw{0}{-2}^{\oplus {3}} \To \Omega_{\pi} \To 0.\\
\end{equation}\\


We end this subsection by recalling the notion of regularity given in \cite{MS} adapted to our case i.e. for the scrolls $X$ with $m=2$ and $n=1$. We will later show  as a consequence of our study of cohomological  properties of Ulrich bundles on $X$ that Ulrich bundles on $X$ are regular in the sense of this definition.  For $p,\,q\in\ZZ$, we set $\bdp=p\Hyp+q\fib$ and we write
$\cF(\bdp)=\cF\tw{p}{q}$.\

\begin{definition}\label{def:pqreg3}
A coherent sheaf $\cF$ on $X$ is said to be \emph{$(p,q)$-regular} if
\begin{enumerate}[(a)]
\item $h^{1}(\cF(\bdp)\tw{-1}{c-1})=h^{2}(\cF(\bdp)\tw{-1}{c-2})=h^{3}(\cF(\bdp)\tw{-1}{c-3})= 0$, and
\item $h^{1}(\cF(\bdp)\tw{0}{-1})=h^{2}(\cF(\bdp)\tw{0}{-2})=0$.
\end{enumerate}
We will say \emph{regular} to mean $(0,0)$-regular.  We define the
\emph{regularity} of $\cF$, denoted $\Reg(\cF)$, to be the least
integer $p$ such that $\cF$ is $(p,0)$-regular. We set
$\Reg(\cF)=-\infty$ if there is no such integer.\ 
\end{definition}



\subsection{Cohomological vansihing properties}\label{sect:splitting}

 In this subsection, for an Ulrich bundle $\cE$ of arbitrary rank on $X$, we describe the complete list of cohomology vanishings of twisted $\cE$ and $\cE \otimes \Omega_{\pi}$, which is crucial for preparing the Beilinson table for obtaining resolutions of Ulrich bundles on $X$ in the next section.\\

\begin{proposition}\label{riv2}(Cohomological vanishing properties of Ulrich bundles)\
    
Let $\cE$ be an Ulrich bundle on $X$. Let $k \in \{0,1,2,3\}$. Then the following holds:\
\medbreak
\begin{enumerate}
\item[(a)]
$
\begin{cases}
H^{3}(\cE\tw{-k}{t})=0 \ \text{for $t \geq (k-3)a_0$},\\
H^{3}(\cE \otimes \Omega_{\pi} \tw{-k}{t})=0 \ \text{for $t \geq 2+(k-3)a_0$}.
\end{cases}$
\medbreak
\item[(b)] For each $k \in \{0,1,2\}$,\\
$\begin{cases}
H^{2}(\cE\tw{-k}{t})=0 \ \text{for $t \geq (k-2)a_0$},\\
H^{2}(\cE \otimes \Omega_{\pi} \tw{-k}{t})=0 \ \text{for $t \geq 2+(k-2)a_0$},\\
H^{2}(\cE\tw{-3}{t})=0 \ \text{for $t \geq c$ and $t \leq 0$}.\\
H^{2}(\cE \otimes \Omega_{\pi} \tw{-3}{t})=0  \ \text{for $t \geq 2+c$ and $t \leq 1$}.
\end{cases}$
\medbreak
\item[(c)] For each $k \in \{1,2,3\}$,\\
$
\begin{cases}
H^{1}(\cE\tw{-k}{t})=0 \ \text{for $t \leq (k-2)c$ and $k \neq 3$},\\
H^{1}(\cE\tw{-3}{t})=0 \ \text{for $t \leq a_0$},\\
H^{1}(\cE \otimes \Omega_{\pi} \tw{-k}{t})=0 \ \text{for $t \leq (k-2)a_{0}+1$ and $k \neq 1$}.\
\end{cases}
$
\medbreak
For each $k \in \{0, 1\}$,\\
$
\begin{cases}
H^{1}(\cE\tw{-k}{t})=0 \ \text{for $t \geq (k-1)a_0$},\\
H^{1}(\cE \otimes \Omega_{\pi} \tw{-k}{t})=0  \ \text{for $t \geq 2+(k-1)a_0$},\\
\end{cases}
$
\medbreak
\item[(d)] For each $k \in \{1,2,3,4\}$,\\
$
\begin{cases}
H^{0}(\cE\tw{-k}{t})=0 \ \text{for $t \leq (k-1)a_0$},\\
H^{0}(\cE \otimes \Omega_{\pi} \tw{-k}{t})=0  \ \text{for $t \leq 1+(k-1)a_{0}$ },\\
H^{0}(\cE \otimes \Omega_{\pi} \tw{0}{t})=0  \ \text{for $t \leq 1-c$}.
\end{cases}
$
\end{enumerate}


\end{proposition}

\begin{proof}

Let $\cE$ be an Ulrich bundle on $X$.\

$(a)$  Let $k \in \{0,1,2,3\}$. Tensorizing the sequence \eqref{eq:34} by $\cE\tw{-3}{t+1}$ and taking cohomology, we get the surjection, $H^{3}(\cE\tw{-3}{t})^{\oplus 3} \To H^{3}(\cE\tw{-3}{t+1}) \To 0$. If $t=0$, the first term vanishes as $\cE$ is Ulrich and hence the next term vanishes. By recursion on $t$, we get $H^{3}(\cE\tw{-3}{t})=0$ for $t \geq 0$. Tensorizing the sequence \eqref{eq:33} by $\cE\tw{-3}{t}$ and taking cohomology, we get the surjection, $\bigoplus_{i=0}^1 H^{3}(\cE\tw{-3}{a_{i}+t}) \To H^{3}(\cE\tw{-2}{t}) \To 0$. If $t \geq -a_0$, then the first term vanishes and hence the next term vanishes. Therefore, we have $H^{3}(\cE\tw{-2}{t})=0$ for $t \geq -a_0$. Recursion on $k$ will settle the assertion for the remaining values of $k$ (i.e. for $k=0,1$).\ 

Tensorizing the sequence \eqref{eq:35} by $\cE\tw{-k}{t}$ and taking cohomology, we get the surjection, $H^{3}(\cE\tw{-k}{t-2})^{\oplus 3} \To H^{3}(\cE \otimes \Omega_{\pi} \tw{-k}{t}) \To 0$. If $t \geq 2+(k-3)a_0$, then by the deduction in the previous paragraph, the first term vanishes and hence the next term vanishes and we are through.\\

$(b)$  Let $k \in \{0,1,2\}$. Tensorizing the sequence \eqref{eq:34} by $\cE\tw{-2}{1}$ and taking cohomology, we have the following piece of long exact sequence,
\begin{equation}
  H^{2}(\cE\tw{-2}{0})^{\oplus 3} \To H^{2}(\cE\tw{-2}{1}) \To H^{3}(\cE \otimes \Omega_{\pi} \tw{-2}{1})
\end{equation}
The first term vanishes since $\cE$ is Ulrich and the third term vanishes by the second item of $(a)$. Therefore, $H^{2}(\cE \tw{-2}{1})=0$. By recursion on $t$, we get $H^{2}(\cE \tw{-2}{t})=0$ for $t \geq 0$. Next, tensorizing the sequence \eqref{eq:33} by $\cE\tw{-2}{t}$ and taking cohomology, we have the following piece of long exact sequence,
\begin{equation}
 \bigoplus_{i=0}^1 H^{2}(\cE\tw{-2}{a_{i}+t}) \To H^{2}(\cE\tw{-1}{t}) \To H^{3}(\cE\tw{-3}{t+c})
\end{equation}
If $t \geq -a_0$, then both the first and the third term vanishes and hence the middle term vanishes. Therefore, we have  $H^{2}(\cE \tw{-1}{t})=0$ for $t \geq -a_0$. Recursion on $k$ will settle the first item of $(b)$ for the remaining value of $k$ (i.e. for $k=0$).\

Tensorizing the sequence \eqref{eq:35} by $\cE\tw{-k}{t}$ and taking cohomology, we have the following piece of long exact sequence, 
\begin{equation}
H^{2}(\cE\tw{-k}{t-2})^{\oplus 3} \To H^{2}(\cE \otimes \Omega_{\pi} \tw{-k}{t}) \To   H^{3}(\cE\tw{-k}{t-3}) 
\end{equation}
 If $t  \geq 2+(k-2)a_0$, the first term vanishes and since this also means $t  \geq 3+(k-3)a_0$, the third term vanishes. Therefore, the middle term vanishes i.e. $H^2(\cE \otimes \Omega_{\pi}\tw{-k}{t})=0$ for $t \geq (k-2)a_0+1$. This settles the second item of $(b)$.\

 If we can show that for an Ulrich bundle $\nu$ on $X$, $H^{1}(\nu\tw{-1}{t})=0$ for $t \geq 0$ and $t \leq -c$ (this will be shown in the proof of the items of $(c)$ independently), then by Serre duality and \cite{UB}, Proposition $3.3.3(ii)$,  we get $H^{2}(\cE\tw{-3}{t})=0$ for $t \geq c$ and $t \leq 0$. This settles the third item of $(b)$.\
 

 Tensorizing the sequence \eqref{eq:35} by $\cE\tw{-3}{t}$ and taking cohomology, we have the following piece of long exact sequence, 
\begin{equation}
H^{2}(\cE\tw{-3}{t-2})^{\oplus 3} \To H^{2}(\cE \otimes \Omega_{\pi} \tw{-3}{t}) \To   H^{3}(\cE\tw{-3}{t-3}) 
\end{equation}
If $t  \geq 2+c$, the first term vanishes by the third item of $(b)$ (which as mentioned earlier will be proven independently in the proof of $(c)$) and since this also means $t-3  \geq 0$, by the first item of $(a)$  the third term vanishes. This means the middle term also vanishes and hence $H^{2}(\cE \otimes \Omega_{\pi} \tw{-3}{t})=0$ for $t \geq 2+c$. Note that, by Serre duality, we have, $H^{2}(\cE \otimes \Omega_{\pi} \tw{-3}{t}) \cong H^{1}(\nu \otimes \Omega_{\pi}^{*} \tw{-1}{-t})$, where $\nu$ is the Ulrich dual of $\cE$ in the sense of \cite{UB}, Proposition $3.3.3(ii)$ and in particular an Ulrich bundle on $X$. Tensorizing the dual of the sequence \eqref{eq:34} by $\nu\tw{-1}{-t}$ and taking cohomology, we have the piece of long exact sequence,
\begin{equation}
H^{1}(\nu\tw{-1}{1-t})^{\oplus 3} \To H^{1}(\nu \otimes \Omega_{\pi}^{*} \tw{-1}{-t}) \To   H^{2}(\nu\tw{-1}{-t}) 
\end{equation}
If $t \leq1$, since $\nu$ is Ulrich, the first term vanishes because of the fourth item of  $(c)$ (which will be proven independently later) and the third term vanishes because of the first item of $(b)$. Therefore, the middle term vanishes and hence $H^{2}(\cE \otimes \Omega_{\pi} \tw{-3}{t})=0$  for $t \leq 1$. This settles the fourth item of $(b)$.\\

$(c)$ If $\nu$ is the Ulrich dual of $\cE$ (in the sense of \cite {UB}, Proposition $3.3.3 (ii)$), then we have, $H^{1}(\cE\tw{-2}{t}) \cong H^{2}(\nu\tw{-2}{-t})$. If $t \leq 0$, since $\nu$ is Ulrich, by the first item of $(b)$, we have the later cohomology vanishes. Therefore, we have, $H^{1}(\cE\tw{-2}{t})=0$ for $t \leq 0$. Next, tensorizing the sequence \eqref{eq:33} by $\cE\tw{-2}{t}$ and taking cohomology, we have the following piece of the long exact sequence,
\begin{equation}
\bigoplus_{i=0}^1 H^{1}(\cE\tw{-2}{a_{i}+t}) \To H^{1}(\cE\tw{-1}{t}) \To H^{2}(\cE\tw{-3}{c+t}).
\end{equation}
If $t \leq -c$, then the first term vanishes by the just obtained $H^1$ vanishing. By Serre duality,  the third term is isomorphic to $H^1(\nu\tw{-1}{-c-t})$, where $\nu$ is the the Ulrich dual of $\cE$ in the sense of \cite {UB}, Proposition $3.3.3 (ii)$ (and hence Ulrich), which vanishes due to the fourth item of $(c)$ (which will be shown independently). Hence the middle term vanishes i.e. $H^{1}(\cE\tw{-1}{t}) =0$ for $t \leq -c$. This settles the first item of $(c)$.\


Note that by Serre duality, $H^{1}(\cE\tw{-3}{t}) \cong H^{2}(\nu\tw{-1}{-t})$. If $t \leq a_0$, the later cohomology vanishes by virtue of the first item of $(b)$. This settles the second item of $(c)$.\

Let $k \in \{2,3\}$. Note that by Serre duality $H^{1}(\cE \otimes \Omega_{\pi} \tw{-k}{t}) \cong H^{2}(\nu \otimes \Omega_{\pi}^{*} \tw{k-4}{-t})$, where $\nu$ is the Ulrich dual of $\cE$ in the sense of \cite{UB}, Proposition $3.3.3(ii)$ (and hence is Ulrich). Therefore, our work boils down to establishing the vanishing of $H^{2}(\nu \otimes \Omega_{\pi}^{*} \tw{k-4}{-t})$. Tensorizing the dual of the sequence \eqref{eq:34} by $\nu\tw{k-4}{-t}$ and taking cohomology, we have the following piece of long exact sequence,
\begin{equation}
  H^{2}(\nu\tw{k-4}{1-t})^{\oplus 3}  \To H^{2}(\nu \otimes \Omega_{\pi}^{*} \tw{k-4}{-t}) \To  H^{3}(\nu\tw{k-4}{-t})
\end{equation}
If $t \leq (k-2)a_0+1$, since $\nu$ is Ulrich, the first term vanishes due to the first item of $(b)$ and the third term vanishes due to the first item of $(a)$. Therefore, the middle term vanishes i.e. we get $H^{1}(\cE \otimes \Omega_{\pi} \tw{-k}{t})=0$ for $t \leq (k-2)a_0+1$. This settles the third item of $(c)$.\

For the fourth item of $(c)$, we proceed as follows: let $k \in \{0,1\}$. Tensorizing the sequence \eqref{eq:34} by $\cE\tw{-1}{t+1}$ and taking cohomology, we have the following piece of long exact sequence,
\begin{equation}
  H^{1}(\cE\tw{-1}{t})^{\oplus 3} \To H^{1}(\cE\tw{-1}{t+1}) \To H^{2}(\cE \otimes \Omega_{\pi} \tw{-1}{t+1})
\end{equation}
If $t=0$, the first term vanishes since $\cE$ is Ulrich and the third term vanishes due to the second item of $(b)$ and hence the middle term vanishes i.e. $H^{1}(\cE\tw{-1}{1})=0$. By recursion on $t$, we get $H^{1}(\cE\tw{-1}{t})=0$ for $t \geq 0$. Tensorizing the sequence \eqref{eq:33} by $\cE\tw{-1}{t}$ and taking cohomology, we have the following piece of the long exact sequence,
\begin{equation}
\bigoplus_{i=0}^1 H^{1}(\cE\tw{-1}{a_{i}+t}) \To H^{1}(\cE\tw{0}{t}) \To H^{2}(\cE\tw{-2}{c+t}).
\end{equation}
If $t \geq -a_0$, then the first term vanishes by the just obtained $H^1$ vanishing and the third term vanishes by the first item of $(b)$. Hence the middle term vanishes i.e. $H^{1}(\cE\tw{0}{t})=0$ for $t \geq -a_0$. This settles the fourth item of $(c)$.\



For the fifth item of $(c)$, we proceed as follows: let $k \in \{0,1\}$. Tensorizing the sequence \eqref{eq:35} by $\cE\tw{-k}{t}$ and taking cohomology, we have the following piece of long exact sequence, 
\begin{equation}
H^{1}(\cE\tw{-k}{t-2})^{\oplus 3} \To H^{1}(\cE \otimes \Omega_{\pi} \tw{-k}{t}) \To   H^{2}(\cE\tw{-k}{t-3}) 
\end{equation}

If $t  \geq 2+(k-1)a_0$, then the first term vanishes due to the fourth item of $(c)$ and the third term vanishes beacuse of the first item of $(b)$. This settles the fifth item of $(c)$\\ 


$(d)$ Let $k \in \{1,2,3,4\}$. By Serre duality, $H^{0}(\cE\tw{-k}{t}) \cong H^{3}(\nu\tw{-(4-k)}{-t}) $, where $\nu$ is the Ulrich dual  of $\cE$ in the sense of \cite{UB}, Proposition $3.3.3(ii)$ (and hence $\nu$ is Ulrich). If $t \leq (k-1)a_0$, since $\nu$ is Ulrich, the later cohomology vanishes due to the first item of $(a)$. This settles the first item of $(d)$.\

By Serre duality, $H^{0}(\cE \otimes \Omega_{\pi}\tw{-k}{t}) \cong H^{3}(\nu \otimes \Omega_{\pi}^{*}\tw{-(4-k)}{-t}) $, where $\nu$ is the Ulrich dual  of $\cE$ in the sense of \cite{UB}, Proposition $3.3.3(ii)$ (and hence $\nu$ is Ulrich). Tensorizing the dual of the sequence  \eqref{eq:34} by $\nu\tw{-(4-k)}{-t} $ and taking cohomology, we get the surjection, $H^{3}(\nu\tw{-(4-k)}{1-t})^{\oplus 3} \To H^{3}(\nu \otimes \Omega_{\pi}^{*} \tw{-(4-k)}{-t}) \To   0 $.  If $t \leq 1+(k-1)a_0$, then as $\nu$ is Ulrich, the first term vanishes by the first item of $(a)$ and hence the second term vanishes. Therefore, we have $H^{0}(\cE \otimes \Omega_{\pi} \tw{-k}{t})=0$  for $t \leq 1+(k-1)a_0$. This settles the second item of $(d)$.\




For the last item of $(d)$ we  proceed as follows: we will first show that for an Ulrich bundle $\cE$ on $X$, we have $H^{3}(\cE  \tw{-4}{t})=0$  for $t \geq c$. Tensorizing the dual of the sequence \eqref{eq:33} by $\cE\tw{-3}{t}$ and taking cohomology, we have the following piece of long exact sequence, 
\begin{equation}
H^{2}(\cE\tw{-2}{t-c}) \To H^{3}(\cE \tw{-4}{t}) \To   \bigoplus_{i=0}^1 H^{3}(\cE\tw{-3}{t-a_i})
\end{equation}
If $t \geq c$, then the first term vanishes due to the first item of $(b)$ and the third term vanishes due  to the first item of $(a)$. Therefore, the middle term vanishes and we have $H^{3}(\cE  \tw{-4}{t})=0$  for $t \geq c$. With this observation at hand, let us note that by Serre duality, we have $H^{0}(\cE \otimes \Omega_{\pi}\tw{0}{t}) \cong H^{3}(\nu \otimes \Omega_{\pi}^{*}\tw{-4}{-t}) $, where $\nu$ is the Ulrich dual  of $\cE$ in the sense of \cite{UB}, Proposition $3.3.3(ii)$ (and hence $\nu$ is Ulrich). Tensorizing the dual of the sequence  \eqref{eq:34} by $\nu\tw{-4}{-t} $ and taking cohomology, we get the surjection, $H^{3}(\nu\tw{-4}{1-t})^{\oplus 3} \To H^{3}(\nu \otimes \Omega_{\pi}^{*} \tw{-4}{-t}) \To   0 $. If $t \leq 1-c$, then since $\nu$ is Ulrich, the first term vanishes by the just obtained vanishing. Therefore, the second term vanishes and we have  $H^{0}(\cE \otimes \Omega_{\pi}\tw{0}{t})=0$ for $t \leq 1-c$, which settles the last item of $(d)$. \

\end{proof}

\begin{remark}\label{rem}
$(i)$ From Proposition \ref{riv2}, it follows that  Ulrich bundles on a smooth toric threefold with Picard number $2$ are \emph{regular} in the sense of definition \ref{def:pqreg3}. So it is also \emph{$(p,q)$-regular} for non-negative integers $p,q$ and other cohomological vanishings can be obtained. \

$(ii)$  One can use the generalized Hoppe's criterion over polycyclic varieties (see \cite{Gauge}, Theorem $3$, the converse direction) coupled with the property concerning the first Chern class of an Ulrich bundle over any smooth projective variety (See \cite{Tangent}, Lemma $3.2(i)$) to obtain  sharper  bounds for the zeroth and the top Cohomology vanishing for an Ulrich bundle (and its tensorization with $\Omega_{\pi}$) on $X$ as follows:\

\begin{itemize}

    \item For $k \in \{3,4\}$, 
    $\begin{sistema}
    H^{0}(\cE\tw{-k}{t})=0 \ \text{for $t \leq \frac{((2k-3)c+3)c-2(k-1)a_{0}a_{1}}{2c}$},\\

     H^{0}(\cE \otimes \Omega_{\pi}  \tw{-k}{t})=0 \ \text{for $t \leq \frac{((2k-3)c+3)c-2(k-1)a_{0}a_{1}}{2c}+1$},\\
        
    \end{sistema}$

    \item $\begin{sistema}
        H^{0}(\cE \tw{-2}{t})=0  \ \text{for $t \leq \frac{c(c+3)-2a_{0}a_{1}}{2c}$, if $2a_{0}a_{1} \leq a^2_1-a^2_0+3c$},\\

         H^{0}(\cE \otimes \Omega_{\pi} \tw{-2}{t})=0  \ \text{for $t \leq \frac{c(c+3)-2a_{0}a_{1}}{2c}+1$, if $2a_{0}a_{1} \leq a^2_1-a^2_0+3c$},\\

        H^{0}(\cE  \tw{-2}{t})=0  \ \text{for $t \leq a_0$, if $2a_{0}a_{1} > a^2_1-a^2_0+3c.$},\\

        H^{0}(\cE \otimes \Omega_{\pi} \tw{-2}{t})=0  \ \text{for $t \leq a_0+1$, if $2a_{0}a_{1} > a^2_1-a^2_0+3c.$},\\

    \end{sistema}$

    \item $\begin{sistema}
    H^{0}(\cE\tw{-1}{t})=0\ \text{for $t \leq 0$},\\

    H^{0}(\cE \otimes \Omega_{\pi} \tw{-1}{t})=0\ \text{for $t \leq 1$},\\

    \end{sistema}$

    \item For $k \in \{0,1\}$, 
    $\begin{sistema}
    H^{3}(\cE\tw{-k}{t})=0 \ \text{for $t \geq \frac{2(3-k)a_{0}a_{1}-((5-2k)c+3)c}{2c}$},\\

     H^{3}(\cE \otimes \Omega_{\pi} \tw{-k}{t})=0 \ \text{for $t \geq \frac{2(3-k)a_{0}a_{1}-((5-2k)c+3)c}{2c} +2$},\\

    \end{sistema}$

    \item $\begin{sistema}
        H^{3}(\cE \tw{-2}{t})=0  \ \text{for $t \geq \frac{2a_{0}a_{1}-c(c+3)}{2c}$, if $2a_{0}a_{1} \leq a^2_1-a^2_0+3c$},\\

         H^{3}(\cE \otimes \Omega_{\pi} \tw{-2}{t})=0  \ \text{for $t \geq \frac{2a_{0}a_{1}-c(c+3)}{2c}+2$, if $2a_{0}a_{1} \leq a^2_1-a^2_0+3c$},\\

          H^{3}(\cE  \tw{-2}{t})=0  \ \text{for $t \geq -a_0$, if $2a_{0}a_{1} > a^2_1-a^2_0+3c.$},\\

        H^{3}(\cE \otimes \Omega_{\pi} \tw{-2}{t})=0  \ \text{for $t \geq 2-a_0$, if $2a_{0}a_{1} > a^2_1-a^2_0+3c.$},\\

    \end{sistema}$

    \item  $\begin{sistema}
    H^{3}(\cE\tw{-3}{t})=0 \ \text{for $t \geq 0$},\\

    H^{3}(\cE \otimes \Omega_{\pi} \tw{-3}{t})=0 \ \text{for $t \geq 2$}
        
    \end{sistema}.$\

\end{itemize}
We don't record these improvements in Proposition \ref{riv2}, as we don't require the improved bounds for cohomology vanishing to obtain the main resolution of Ulrich bundles on $X$. Also note that this improvements will be  useful in the next section for obtaining  resolution, monadic description of Ulrich bundles on special subclass (i.e. for special values of $c$) of smooth toric threefolds with Picard number $2$. \\
\end{remark}

\section{Resolutions of Ulrich bundles }\label{Res}

Equipped with the cohomological vanishing properties of Ulrich bundles in the previous section, in this section we focus on the main theorem of this paper i.e. construction of a resolution of Ulrich bundles on the threefolds under consideration. Towards this, we consider the following full exceptional collection (see \cite{Orlov}) that we need for working with the Beilinson theorem,

\begin{equation}\label{col5}\begin{aligned}
& \cE_5[k_5]= \cO_X\tw{-1}{c-2}[-2], \cE_4[k_4]=\Omega_{\pi}\tw{-1}{c}[-2], \cE_3[k_3]=\cO_X\tw{-1}{c-1}[-2],\\
& \cE_2[k_2]= \cO_X\tw{0}{-1}, \cE_1[k_1]= \Omega_{\pi}\tw{0}{1}, \cE_0[k_0]= \cO_X.\\
\end{aligned}
\end{equation}


The associated full exceptional collection $\langle F_5=\cF_5, \ldots, F_0=\cF_0\rangle$ of Theorem \ref{use} is

\begin{equation}\label{col6}\begin{aligned}
& \cF_5= \cO_X\tw{-1}{-1}, \cF_4=\cO_X\tw{-1}{0}, \cF_3=\cO_X\tw{-1}{1},\\
& \cF_2= \cO_X\tw{0}{-2}, \cF_1= \cO_X\tw{0}{-1}, \cF_0= \cO_X.\\
\end{aligned}
\end{equation}

For each sheaf $\mathcal G$ and integer $j$ we denote by $\mathcal G[j]$ the complex $G$ such that
$$
 G_i=\left\lbrace\begin{array}{ll} 
0\quad&\text{if $i\ne j$,}\\
\mathcal G \quad&\text{if $i=j$,}
\end{array}\right.
$$
with the trivial differentials: we will omit $[0]$ in the notation.\\

For an Ulrich bundle $\cE$ on $X$, let us also introduce the following notations:
\begin{itemize}
    \item $a_i^{j,k}:=h^{i}(\cE\tw{j}{k})$.\

    \item $b_i^{j,k}:=h^{i}(\cE \otimes \Omega_{\pi}\tw{j}{k})$
\end{itemize}

In what follows, we give a resolution of an Ulrich bundle on smooth toric threefolds $X$ with Picard number $2$.\


\begin{theorem}\label{main} (Resolution of Ulrich bundles on $X$)\

With the notation as above, let $\cE$ be an Ulrich bundle on $X$. Then $\cE$ arises from an exact sequence of the form:


\begin{equation}\label{ww}
0\to \cO_X\tw{-1}{-1}^{\oplus a_0^{-1,c-2}}\to \cO_X\tw{-1}{0}^{\oplus b_0^{-1,c}} \bigoplus \cO_X\tw{0}{-2}^{\oplus a_0^{0,-1}}\to 
\end{equation}
$$ \to \cO_X\tw{-1}{1}^{\oplus a_0^{-1,c-1}} \bigoplus \cO_X\tw{0}{-1}^{\oplus b_0^{0,1}} \to \cO_X^{\oplus a_0^{0,0}} \to \cE \to 0$$
    
\end{theorem}

\begin{proof}

Let $\cE$ be an Ulrich bundle on $X$. We consider the Beilinson type spectral sequence associated to $\cE$
and identify the members of the graded sheaf associated to the induced filtration as the sheaves mentioned in the statement of the Theorem \ref{use}. We consider the full exceptional collection $\cE_{\bullet}$ as in (\ref{col5}) and the right dual collection $\cF_{\bullet}$ as in (\ref{col6}).\
    

We construct a Beilinson complex, quasi-isomorphic to $\cE$, by computing $H^{i+k_j}(\cE\otimes \cE_j)\otimes \cF_j$ with  $i,j \in \{0, \ldots, 5\}$ and the vanishing properties of the Ulrich bundle $\cE$ as obtained in Proposition \ref{riv2} to get the following table:

 \begin{center}\begin{tabular}{|c|c|c|c|c|c|c|c|c|c|c|}
\hline
 $\cO_X\tw{-1}{-1}$ & $\cO_X\tw{-1}{0}$ & $\cO_X\tw{-1}{1}$ & $\cO_X\tw{0}{-2}$ & $\cO_X\tw{0}{-1}$ & $\cO_X$ \\
 \hline
 \hline
$0$	&	$0$	&	$0$		&	$*$		&	$*$		&	$*$	\\
$0$	&	$0$	&	$0$	&	$*$	&	$*$		&	$*$	\\
$0$	& 	$0$	&	$0$	&	$0$	&	$0$	&	$0$	\\
$a_0^{-1,c-2}$	& 	$b_0^{-1,c}$	&	$a_0^{-1,c-1}$	&	$0$	&	$0$	&	$0$	\\
$*$		&	$*$ 	 	&	$*$	&	$0$	& 	$0$	& 	$0$	\\
$*$		&	$*$		&	$*$		&	$a_0^{0,-1}$		&	$b_0^{0,1}$	& 	$a_0^{0,0}$ \\
\hline
$\cO_X\tw{-1}{c-2}$		& $\Omega_{\pi}\tw{-1}{c} $	 & $\cO_X\tw{-1}{c-1}$	 & $\cO_X\tw{0}{-1}$		& $\Omega_{\pi}\tw{0}{1}$		&$\cO_X$\\
\hline
\end{tabular}
\end{center}

Since $Ext^k(F_i,F_j) = 0$ for $k>0$ and any $i,j$, we have that the  collection (\ref{col6}) is strong and hence by the strong form of Beilinson's theorem (as in Remark \ref{rembeil}), we get the claimed resolution.\


\end{proof}

Note that in the above theorem, we have only used Proposition \ref{riv2} for the cohomological vanishing results. Instead if we use the improvements as in Remark \ref{rem} $(ii)$ and carry out the same procedure on $\mathcal E\tw{-1}{0}$ (instead of $\mathcal E$) as in the proof of the Theorem \ref{main}, then for $c \leq 3$, we obtain a Beilinson table, where only the $h^1$ terms survive. To be more precise, we get the following table:

 \begin{center}\begin{tabular}{|c|c|c|c|c|c|c|c|c|c|c|}
\hline
 $\cO_X\tw{-1}{-1}$ & $\cO_X\tw{-1}{0}$ & $\cO_X\tw{-1}{1}$ & $\cO_X\tw{0}{-2}$ & $\cO_X\tw{0}{-1}$ & $\cO_X$ \\
 \hline
 \hline
$0$	&	$0$	&	$0$		&	$*$		&	$*$		&	$*$	\\
$0$	&	$0$	&	$0$	&	$*$	&	$*$		&	$*$	\\
$a_1^{-2, c-2}$	& 	$b_1^{-2,c}$	&	$a_1^{-2,c-1}$	&	$0$	&	$0$	&	$0$	\\
$0$	& 	$0$	&	$0$	&	$0$	&	$0$	&	$0$	\\
$*$		&	$*$ 	 	&	$*$	&	$a_1^{-1,-1}$	& 	$b_1^{-1,1}$	& 	$0$	\\
$*$		&	$*$		&	$*$		&	$0$		&	$0$	& 	$0$ \\
\hline
$\cO_X\tw{-1}{c-2}$		& $\Omega_{\pi}\tw{-1}{c} $	 & $\cO_X\tw{-1}{c-1}$	 & $\cO_X\tw{0}{-1}$		& $\Omega_{\pi}\tw{0}{1}$		&$\cO_X$\\
\hline
\end{tabular}
\end{center}

This means for $c \leq 3$, $\mathcal E\tw{-1}{0}$ fits into an exact sequence of the form:

\begin{equation*}
0\to \cO_X\tw{-1}{-1}^{\oplus a_1^{-2, c-2}} \to \cO_X\tw{-1}{0}^{\oplus b_1^{-2,c}} \bigoplus \cO_X\tw{0}{-2}^{\oplus a_1^{-1,-1}} \to \\
 \end{equation*}   

 $$\to \cO_X\tw{-1}{1}^{\oplus a_1^{-2,c-1}}  \bigoplus  \cO_X\tw{0}{-1}^{\oplus b_1^{-1,1}} \to \mathcal E\tw{-1}{0} \to 0.$$\


We summarize the above discussion in the following Proposition.

\begin{proposition}\label{lowc}(Resolution for low values of $c$)\\
With the notation as above, let $\cE$ be an Ulrich bundle on $X$. If $c \leq 3$, then $\cE$ fits into an exact sequence of the form:

\begin{equation}
0\to \cO_X\tw{0}{-1}^{\oplus a_1^{-2, c-2}} \to  \cO_X^{\oplus b_1^{-2,c}} \bigoplus \cO_X\tw{1}{-2}^{\oplus a_1^{-1,-1}} \to 
\end{equation}
$$\to \cO_X\tw{0}{1}^{\oplus a_1^{-2,c-1}} \bigoplus  \cO_X\tw{1}{-1}^{\oplus b_1^{-1,1}} \to \mathcal E \to 0.$$
    
\end{proposition}

Note that, instead if we carry out the same procedure on $\cE\tw{-2}{0}$ (instead of $\cE$ and $\cE\tw{-1}{0}$) as in the proof of the Theorem \ref{main} and use the improvements as in Remark \ref{rem}$(ii)$, then we obtain a Beilinson table, where possibly the $h^1$ and  $h^2$ terms survive. To be more precise, we get the following table:

 \begin{center}\begin{tabular}{|c|c|c|c|c|c|c|c|c|c|c|}
\hline
 $\cO_X\tw{-1}{-1}$ & $\cO_X\tw{-1}{0}$ & $\cO_X\tw{-1}{1}$ & $\cO_X\tw{0}{-2}$ & $\cO_X\tw{0}{-1}$ & $\cO_X$ \\
 \hline
 \hline
$0$	&	$0$	&	$0$		&	$*$		&	$*$		&	$*$	\\
$a_2^{-3, c-2}$	&	$b_2^{-3,c}$	&	$a_2^{-3,c-1}$	&	$*$	&	$*$		&	$*$	\\
$a_1^{-3, c-2}$	& 	$b_1^{-3,c}$	&	$a_1^{-3,c-1}$	&	$0$	&	$0$	&	$0$	\\
$0$	& 	$0$	&	$0$	&	$a_2^{-2,-1}$	&	$b_2^{-2,1}$	&	$0$	\\
$*$		&	$*$ 	 	&	$*$	&	$0$	& 	$0$	& 	$0$	\\
$*$		&	$*$		&	$*$		&	$0$		&	$0$	& 	$0$ \\
\hline
$\cO_X\tw{-1}{c-2}$		& $\Omega_{\pi}\tw{-1}{c} $	 & $\cO_X\tw{-1}{c-1}$	 & $\cO_X\tw{0}{-1}$		& $\Omega_{\pi}\tw{0}{1}$		&$\cO_X$\\
\hline
\end{tabular}
\end{center}

If $a_0=a_1$ and $a_1^{-3,c-1}=0$, then one can use \cite{UB}, Proposition $3.3.3(ii)$, S.E.S \eqref{eq:34} and Proposition \ref{riv2} $(a)$ first item to deduce that $a_1^{-3,c-2}=0$. Also in this situation, from the S.E.S \eqref{eq:34} and Proposition \ref{riv2} $(d)$ first item it follows that $b_1^{-3,c}=0$. Note that if $a_0=a_1=1$, then we automatically have our starting assumption $a_1^{-3,c-1}=0$ by Proposition \ref{riv2}$(c)$ second item.\

On the other hand, if $c=3$ i.e. $(a_0,a_1)=(1,2)$ and $a_1^{-3,c-1}=0$, then by Proposition \ref{riv2}$(c)$ second item we get $a_1^{-3,c-2}=0$. Note that by the first item of Remark \ref{rem}, $(ii)$, we have $h^0(\cE\tw{-3}{t})=0$ for $t \leq 4$. Then from the S.E.S \eqref{eq:34} we get $b_1^{-3,c}=0$.\

Therefore, in these two situations, we get that $\cE\tw{-2}{0}$ is the homology of the monad :
\begin{equation*}
    0 \to \mathcal A \to \mathcal B \to \mathcal C \to 0,
\end{equation*}
where, $\mathcal A:=  \cO_X\tw{-1}{-1}^{\oplus a_2^{-3, c-2}}$,
$\mathcal B :=\cO_X\tw{-1}{0}^{\oplus b_2^{-3, c}}  \bigoplus \cO_X\tw{0}{-2}^{\oplus a_2^{-2, -1}}$ 
and $\mathcal C:= \cO_X\tw{-1}{1}^{\oplus a_2^{-3, c-1}} \bigoplus \cO_X\tw{0}{-1}^{\oplus b_2^{-2, 1}}$.\\


We summarize the above discussion in the following Remark.\


\begin{remark}\label{monad}(monadic description of certain Ulrich bundles on $X$)\\
With the notation as above, let $a_0=a_1$ or $c=3$. If $\cE$ is an Ulrich bundle on $X$ with $h^{1}(\cE\tw{-3}{c-1})=0$. Then $\cE$ is the homology of a monad:


\begin{align*}
    0 \to \cO_X\tw{1}{-1}^{\oplus a_2^{-3, c-2}}  &\to \cO_X\tw{1}{0}^{\oplus b_2^{-3, c}} \bigoplus \cO_X\tw{2}{-2}^{\oplus a_2^{-2, -1}} \to \\ & \to \cO_X\tw{1}{1}^{\oplus a_2^{-3, c-1}} \bigoplus \cO_X\tw{2}{-1}^{\oplus b_2^{-2, 1}} 
    \to 0.
\end{align*}

\end{remark}

\section{Examples and Ulrich Wildness}\label{S4}

In this section, we give examples of Ulrich bundles on $X$ and establish the Ulrich wildness of $X$. Towards this end, we use the study of Ulrich bundles on Veronese surfaces (as in \cite{Genc} and \cite{Costa}) to obtain a complete characterization of the vector bundles $G$ on $\mathbb P^2$ such that $G_\pi\tw{a}{b}:=\pi^*(G)\tw{a}{b}$ is Ulrich on $X$ for some $(a,b) \in \mathbb Z^2$. We will see that as a consequence of this result, we get a classification of Ulrich line bundles on $X$, Ulrichicity of twisted $\Omega_\pi$ on $X$ and the Ulrich wildness of $X$.\\

For any positive integer $d$, let us  denote $(\mathbb P^2, dH)$ to be the Veronese  surface of degree $d^2$. Then we prove the following:\

\begin{proposition}\label{pullback} (Classification of twisted Ulrich Pullbacks)\

 $G_\pi\tw{a}{b}$ is Ulrich on $X$ for some $G$ on $\mathbb P^2$ and $(a,b) \in \mathbb Z^2$ if and only if one of the following conditions is satisfied:

$(i)$  $a=0$, $a_0=a_1$, $b=c-a_0$ and $G$ is Ulrich on $(\mathbb P^2, a_0H)$.\



$(ii)$  $a=1$, $b=-c$ and  $G$ is Ulrich on $(\mathbb P^2, cH)$.\

$(iii)$ $a=2$, $a_0=a_1$, $b=-2a_0$ and $G$ is Ulrich on $(\mathbb P^2, a_0H)$.\



\end{proposition}

\begin{proof}
Let $G_\pi\tw{a}{b}$ be Ulrich on $X$. If we denote $F=G(b)$, then $F_\pi\tw{a}{0}$ is Ulrich on $X$. By \cite {UB}, Proposition $3.3.3 (ii)$, $(F^{*})_{\pi}\tw{2-a}{c-3}$ is also Ulrich. Therefore, $h^0(X, F_\pi\tw{a}{0}) \neq 0, h^0(X, (F^{*})_{\pi}\tw{2-a}{c-3}) \neq 0$ and hence by Lemma \ref{lem:vanishingbundle} it follows that $a \in \{0,1,2\}$. In what follows we analyze each of these three cases separately.\\

\underline{\textbf{Case-$(i)$: $a=0$}}


Let $G_\pi\tw{0}{b}$ be Ulrich on $X$, i.e.  $H^i(X, G_\pi\tw{-1}{b})=0, H^i(X, G_\pi\tw{-2}{b})=0$ and $H^i(X, G_\pi\tw{-3}{b})=0$ for all $i \in \{0,1,2,3\}$. This means from Lemma \ref{lem:vanishingbundle}, we have $H^i(\mathbb P^2, G(b-c))=0, H^i(\mathbb P^2, G(b-a_0-c))=0$ and $H^i(\mathbb P^2, G(b-a_1-c))=0$ for all $i \in \{0,1,2\}$. Since $a_0 >0$, we must have $a_0=a_1$. This means, $G(b+a_0-c)$ is an Ulrich bundle on $(\mathbb P^2, a_0H)$. Therefore,  $b=c-a_0$ and $G$ is an Ulrich bundle on $(\mathbb P^2, a_0H)$ and condition $(i)$ is achieved. \\


\underline{\textbf{Case-$(ii)$: $a=1$}}

Let $G_\pi\tw{1}{b}$ be Ulrich on $X$, i.e.  $H^i(X, G_\pi\tw{0}{b})=0, H^i(X, G_\pi\tw{-1}{b})=0$ and $H^i(X, G_\pi\tw{-2}{b})=0$ for all $i \in \{0,1,2,3\}$. This means from Lemma \ref{lem:vanishingbundle}, we have $H^i(\mathbb P^2, G(b))=0$ and $H^i(\mathbb P^2, G(b-c))=0$ for all $i \in \{0,1,2\}$. Therefore, $G(b+c)$ is Ulrich on $(\mathbb P^2, cH)$. Therefore, $b=-c$ and $G$ is Ulrich on $(\mathbb P^2, cH)$ and condition $(ii)$ is achieved.\\


\underline{\textbf{Case-$(iii)$: $a=2$}}

Let $G_\pi\tw{2}{b}$ be Ulrich on $X$, i.e.  $H^i(X, G_\pi\tw{1}{b})=0, H^i(X, G_\pi\tw{0}{b})=0$ and $H^i(X, G_\pi\tw{-1}{b})=0$ for all $i \in \{0,1,2,3\}$. This means from Lemma \ref{lem:vanishingbundle}, we have $H^i(\mathbb P^2, G(b))=0, H^i(\mathbb P^2, G(b+a_0))=0$ and $H^i(\mathbb P^2, G(b+a_1))=0$ for all $i \in \{0,1,2\}$. Since $a_0 >0$, we must have $a_0=a_1$. This means, $G(b+2a_0)$ is an Ulrich bundle on $(\mathbb P^2, a_0H)$. Therefore, $b=-2a_0$ and $G$ is Ulrich on $(\mathbb P^2, a_0H)$ and condition $(iii)$ is achieved.\\



For the converse assertion, in each case,  one can use the assumption of Ulrichicity of $G$ on some Veronese surface and \cite{Genc}, Proposition $2.1$, to see that $G^*(s)$ (for appropriate $s$) is also Ulrich there.  Then, one can use \cite{Genc}, Corollary $4.3$, to observe that $G$ fits into two S.E.S's (one as kernel and other one as cokernel). Then proving that the corresponding $G_\pi\tw{a}{b}$ is Ulrich is a straightforward cohomology computation in each case using these S.E.S's and Lemma \ref{lem:vanishing}.\


\end{proof}

\begin{remark}
Let us consider the following notations:  $\alpha_i^{j}:=h^{i}(\mathbb P^2, G(j))$ and  $\beta_i^{j}:=h^{i}(\mathbb P^2, G \otimes \Omega^1_{\mathbb P^2}(j))$.   If $G_{\pi}$ is Ulrich on $X$ for some vector bundle $G$ on $\mathbb P^2$, then one can use Lemma \ref{lem:vanishingbundle} and Proposition \ref{riv2}, to produce a list of vanishings for $G(t)$ and $ G \otimes \Omega^1_{\mathbb P^2}(t)$. Then we apply spectral sequence of \cite{okonek}, Theorem $3.1.4$ (Beilinson, Theorem $(II)$) to $G$ and deduce that $G$ fits into an exact sequence of the form:
\begin{equation}
    0 \to \mathcal O_{\mathbb P^2}(-2)^{\oplus \alpha_0^{-1}} \to \mathcal O_{\mathbb P^2}(-1)^{\oplus \beta_0^{1}} \to \mathcal O_{\mathbb P^2}^{\oplus \alpha_0^{0}} \to G \to 0
\end{equation}
Observe that this sequence on $\mathbb P^2$ is aligned with the the sequence \eqref{ww} of our main Theorem \ref{main} on $X$.\

\end{remark}

Next, we give several important corollaries of the above Proposition.\

\begin{corollary}\label{line}(Classification of Ulrich line bundles)\

Let $\mathcal O_X\tw{a}{b}$ is Ulrich on $X$ if and only if one of the following happens:\

$(i)$ $(a_0,a_1,a,b)=(1,1,0,1)$.\

$(ii)$ $(a_0,a_1,a,b)=(1,1,2,-2)$.\

\end{corollary}

\begin{proof}
Let $\mathcal O_X\tw{a}{b}$ is Ulrich on $X$. Note that this is a special case of Proposition \ref{pullback} with $G=\mathcal O_{\mathbb P^2}$.  Then we consider three cases as follows:
\begin{itemize}
    \item  Let $a=0$. By Proposition \ref{pullback} $(i)$, we have $a_0=a_1$, $b=a_0$ and $\mathcal O_{\mathbb P^2}$ is Ulrich on ($\mathbb P^2, a_0H)$. Then \cite{Genc}, Proposition $3.1$ forces $a_0=1$. Therefore, we have $(a_0,a_1,a,b)=(1,1,0,1)$ and condition $(i)$ is achieved.\

    

    \item Let $a=1$. By Proposition \ref{pullback} $(ii)$, we have $\mathcal O_{\mathbb P^2}$ is Ulrich on $(\mathbb P^2, cH)$. Since $c \geq 2$, this case can't occur by \cite{Genc}, Proposition $3.1$.\
    

     \item Let $a=2$. By Proposition \ref{pullback} $(iii)$, we have $a_0=a_1$, $b=-2a_0$ and $\mathcal O_{\mathbb P^2}$ is Ulrich on ($\mathbb P^2, a_0H)$. Then \cite{Genc}, Proposition $3.1$ forces $a_0=1$. Therefore, we have $(a_0,a_1,a,b)=(1,1,2,-2)$ and condition $(iii)$ is achieved.\

     
\end{itemize}

The converse assertion is a straightforward cohomology computation using the Lemma \ref{lem:vanishing}.\

\end{proof}

\begin{remark}
We point out that the classification of Ulrich line bundles on \(X\) also appears as a special case of the study of Ulrich line bundles on threefold scrolls of low rank carried out in \cite{U3}, see Theorem~0.1, Proposition~5.1, and Remark~5.2. Their approach relies on explicit computations on \(\mathbb{P}^2\). In contrast, our classification follows naturally from understanding Ulrich pullbacks of arbitrary rank on \(X\), together with the study of Ulrich bundles on Veronese surfaces developed in \cite{Genc}.
\end{remark}

\begin{corollary}\label{cotangent}(Classification of twisted, pulled back cotangent bundle)\

Let $\Omega_{\pi} \tw{a}{b}$ is Ulrich on $X$ if and only if one of the following occurs:\

$(i)$ $(a_0,a_1,a,b)=(2,2,0,5)$.\

$(ii)$ $(a_0,a_1,a,b)=(1,1,1,1)$.\

$(iii)$  $(a_0,a_1,a,b)=(2,2,2,-1)$.\
   
\end{corollary}

\begin{proof}

Let $\Omega_{\pi} \tw{a}{b}$ is Ulrich on $X$. Note that this is a special case of Proposition \ref{pullback} with $G=\Omega^{1}_{\mathbb P^2}(3)$.  In this case, $\Omega_{\pi} \tw{a}{b}$ is same as $G_{\pi}\tw{a}{b-3}$. Then we consider three cases as follows:

\begin{itemize}
    \item Let $a=0$. By Proposition \ref{pullback} $(i)$, we have $a_0=a_1$, $b-3=a_0$ and $\Omega^{1}_{\mathbb P^2}(3)$ is Ulrich on ($\mathbb P^2, a_0H)$. One can use \cite{Genc}, Remark $5.3$, to see that $\Omega^{1}_{\mathbb P^2}(3)$ is the unique Ulrich bundle of rank $2$ on $(\mathbb P^2, 2H)$ and  $\Omega^{1}_{\mathbb P^2}(3)$ is never Ulrich on $(\mathbb P^2, dH)$ for $d \neq 2$. This  forces $a_0=2$ and hence  $b=5$. Therefore, we have $(a_0,a_1,a,b)=(2,2,0,5)$ and condition $(i)$ is achieved.\



    \item Let $a=1$. By Proposition \ref{pullback} $(ii)$, we have, $b-3=-c$ and $\Omega^{1}_{\mathbb P^2}(3)$ is Ulrich on ($\mathbb P^2, cH)$. One can use \cite{Genc}, Remark $5.3$, to see that $\Omega^{1}_{\mathbb P^2}(3)$ is the unique Ulrich bundle of rank $2$ on $(\mathbb P^2, 2H)$ and  $\Omega^{1}_{\mathbb P^2}(3)$ is never Ulrich on $(\mathbb P^2, dH)$ for $d \neq 2$. This  forces $c=2$ and $b=1$. Therefore, we have $(a_0,a_1,a,b)=(1,1,1,1)$ and condition $(ii)$ is achieved.\



    \item Let $a=2$. By Proposition \ref{pullback} $(iii)$, we have $a_0=a_1$, $b-3=-2a_0$ and $\Omega^{1}_{\mathbb P^2}(3)$ is Ulrich on ($\mathbb P^2, a_0H)$. One can use \cite{Genc}, Remark $5.3$, to see that $\Omega^{1}_{\mathbb P^2}(3)$ is the unique Ulrich bundle of rank $2$ on $(\mathbb P^2, 2H)$ and  $\Omega^{1}_{\mathbb P^2}(3)$ is never Ulrich on $(\mathbb P^2, dH)$ for $d \neq 2$. This  forces $a_0=2$ and $b=-1$. Therefore, we have $(a_0,a_1,a,b)=(2,2,2,-1)$ and condition $(iii)$ is achieved.\
    
\end{itemize}

Conversely, We observe that:
\begin{itemize}
    \item Since $\Omega^{1}_{\mathbb P^2}(3)$ is Ulrich on $(\mathbb P^2, 2H)$. Therefore, by Proposition \ref{pullback} $(i)$, $\Omega_{\pi} \tw{0}{5}$ is Ulrich on $X$ with $a_0=a_1=2$.\


    \item  Since $\Omega^{1}_{\mathbb P^2}(3)$ is Ulrich on $(\mathbb P^2, 2H)$. Therefore, by Proposition \ref{pullback}$(ii)$, $\Omega_{\pi} \tw{1}{1}$ is Ulrich on $X$ with $a_0=a_1=1$.\
    

    \item Since $\Omega^{1}_{\mathbb P^2}(3)$ is Ulrich on $(\mathbb P^2, 2H)$. Therefore, by Proposition \ref{pullback}$(iii)$, $\Omega_{\pi} \tw{2}{-1}$ is Ulrich on $X$ with $a_0=a_1=2$.\
    
\end{itemize}

\end{proof}

\begin{corollary}\label{wild}(Ulrich wildness)\

    Let $X$ be a smooth toric threefold with Picard number $2$. Then $X$ is Ulrich wild.\
\end{corollary}

\begin{proof}
From Proposition \ref{pullback}, we have the following:
\begin{itemize}
    \item $G_\pi\tw{0}{a_0}$ is Ulrich on $X$ if and only if $G$ is Ulrich on $(\mathbb P^2, a_0H)$. (In this case $r.(a_0-1) \equiv 0 \pmod{2} $).\

    \item $G_\pi\tw{1}{-c}$ is Ulrich on $X$ if and only if $G$ is Ulrich on $(\mathbb P^2, cH)$. (In this case $r.(c-1) \equiv 0 \pmod{2} $).\

    \item $G_\pi\tw{2}{-2a_0}$ is Ulrich on $X$ if and only if $G$ is Ulrich on $(\mathbb P^2, a_0H)$. (In this case $r.(a_0-1) \equiv 0 \pmod{2} $).\
\end{itemize}
Also note that, $(\mathbb P^2, H)$ and $(\mathbb P^2, 2H)$ are of finite Ulrich representation type (see \cite{Genc}, Proposition $3.1$ and Theorem $5.4$) and for  $(\mathbb P^2, dH)$ for $d>2$ is of wild representation type (see \cite{Genc}, Theorem $6.1$, Theorem $6.2$ and \cite{Costa}, Theorem $3$). Therefore, it follows that  $X$ is Ulrich wild except $(a_0, a_1)=(1,1)$. If $(a_0,a_1)=(1,1)$, then $X$ is Ulrich wild by \cite{FMS} and \cite{AHMP}.\\
\end{proof}

\begin{remark}\label{nonpullback}Let $e=a_1-a_0$. In order to construct rank-two Ulrich bundles that are not pullbacks, we start from the cases with small $e$ and minimal polarization, i.e. $a_0=1$. In \cite{AHMP},  all the rank two Ulrich bundles have been classified for $e=0$ and in \cite{CMP}, for $e=1$. In \cite{AM2}, the notion of $H$-instanton bundles was given on every three-dimensional polarized projective manifold, and in \cite{GM}, $H$-instantons were constructed on scrolls for every $e\geq 0$ (for the case $e=1$ see also \cite{CCGM} and for $e=2$ \cite{GJ}). Those of minimal charge are Ulrich, while those of greater charge turn out to be Ulrich for suitable polarizations (i.e., suitable $a_0$). In this way, one can obtain Ulrich bundles of rank two that are not pullbacks and, consequently, of every even rank in all the scrolls we are considering.\\

\end{remark}

We end this paper with the following open questions:\\

{\bf Question 1:} How many and what cohomological conditions are necessary to characterize rank two Ulrich bundles (not pullbacks) considered in the above Remark?

{\bf Question 2:} Is it possible to construct Ulrich bundles (not pullbacks) of odd rank?











    

\section{Acknowledgements}
The second named author is a member of the GNSAGA group of INdAM. The first named author was supported by INdAM for his visit to Italy. The first named author thanks the  Dipartimento di Scienze Matematiche, Politecnico di Torino for their warm hospitality during his visit in February, 2025. The first named author is currently supported by the NBHM (DAE) Postdoctoral Fellowship.\

\bibliographystyle{amsplain}

\bigskip
\noindent
Debojyoti Bhattacharya,\\
Chennai Mathematical Institute,\\
H1 SIPCOT IT Park, Siruseri, Kelambakkam 603103,
India\\
e-mail: {\tt debojyoti7054@gmail.com; debojyoti@cmi.ac.in}

\bigskip
\noindent
Francesco Malaspina,\\
Dipartimento di Scienze Matematiche, Politecnico di Torino,\\
Corso Duca degli Abruzzi 24, 10129 Torino, Italy,\\
e-mail: {\tt francesco.malaspina@polito.it}
\end{document}